\documentclass[11pt]{amsart}
\usepackage[T1]{fontenc}
\usepackage{lmodern}
\usepackage[margin=1in]{geometry}
\usepackage{amsmath,amssymb,amsthm,mathtools,booktabs,microtype,needspace,xcolor,mathrsfs}

\usepackage[colorlinks=true,linkcolor=blue!45!black,citecolor=blue!45!black,urlcolor=blue!45!black]{hyperref}

\hypersetup{pdftitle={Algebraically primitive Teichmuller curves in the minimal hyperelliptic stratum},pdfauthor={Myeongjae Lee}}
\newtheorem{theorem}{Theorem}[section]
\newtheorem{proposition}[theorem]{Proposition}
\newtheorem{lemma}[theorem]{Lemma}

\theoremstyle{remark}

\DeclareMathOperator{\Tr}{Tr}

\DeclareMathOperator{\diag}{diag}
\DeclareMathOperator{\Span}{span}
\DeclareMathOperator{\SL}{SL}
\newcommand{\Q}{\mathbb Q}
\newcommand{\R}{\mathbb R}
\newcommand{\C}{\mathbb C}
\newcommand{\Z}{\mathbb Z}
\newcommand{\PP}{\mathbb P}
\newcommand{\hyp}{\Omega\mathcal M_g(2g-2)^{\mathrm{hyp}}}
\newcommand{\HH}{\mathbb H}

\numberwithin{equation}{section}
\title[Teichm\"uller curves in the minimal hyperelliptic stratum]
{Algebraically primitive Teichm\"uller curves in \(\hyp\)}
\author{Myeongjae Lee}
\date{}
\begin{document}
\begin{abstract}
We give a complete classification of algebraically primitive Teichm\"uller curves in $\Omega\mathcal M_g(2g-2)^{\mathrm{hyp}}$ for $g\ge5$. They are precisely the curves generated by the double regular
\((2g+1)\)-gon when \(2g+1\) is prime, and by the regular
\(4g\)-gon when \(g\) is a power of two.
\end{abstract}
\maketitle
\tableofcontents

\section{Introduction}\label{sec:intro}

Let \(C\) be a Teichm\"uller curve generated by a translation
surface \((X,\omega)\). Its trace field is
\[
 K\coloneqq \Q\bigl(\{\operatorname{tr}(A):A\in\SL(X,\omega)\}\bigr),
\]
where \(\SL(X,\omega)\) is the Veech group. The field \(K\)
is totally real and has degree at most \(g\). The curve is
\emph{algebraically primitive} if \([K:\Q]=g\). In this paper, we study algebraically primitive Teichm\"uller curves in hyperelliptic minimal strata \(\hyp\) for $g>2$. In genus two, the algebraically primitive curves are completely classified by \cite{Cal04,McM03,McM05}: the irreducible components of the Weierstrass curves \(W_D\subset\mathcal{M}_2\) for non-square discriminants $D$. For prime genus $g>2$, Matheus--Wright \cite{MW} proved that there are only finitely many in \(\Omega\mathcal{M}_g(2g-2)\). In this paper, by applying the description of the cusps by M\"oller \cite{Moller}, together with the Harder--Narasimhan filtration of the Hodge bundle given by Yu--Zuo \cite{YZ}, we have the complete classification in \(\hyp\) for any $g\ge 5$.

\begin{theorem}\label{thm:classification}
Let \(g>2\). The following statements hold.
\begin{enumerate}
\item The component \(\hyp\) contains an algebraically primitive
Teichm\"uller curve if and only if \(2g+1\) is prime or \(g\)
is a power of two.
\item If \(g\ge 5 \), every such curve is one of the following:
the double regular \((2g+1)\)-gon curve, when \(2g+1\) is prime;
or the regular \(4g\)-gon curve, when \(g\) is a power of two.
Each polygon type gives exactly one curve.
\end{enumerate}
\end{theorem}

\subsection{Strategy of proof}
\begin{itemize}
    \item In Section~\ref{sec:intro}, we describe the normalized irreducible cusps of the algebraically primitive Teichm\"uller curves.
    \item In Section~\ref{sec:HN-duality}, we deduce the consequences of Harder--Narasimhan filtration and trace duality. In particular, we prove $K/\Q$ is Galois.
    \item In Section~\ref{sec:galois}, we study the group structure of $\operatorname{Gal}(K/\Q)$ to prove the first statement in Theorem~\ref{thm:classification}: $p=2g+1$ is prime or $g$ is a power of two. Also we prove Theorem~\ref{thm:weights} for $g\ge8$. 
    \item In Section~\ref{sec:elliptic}, we find an elliptic element in the Veech group of algebraically primitive Teichm\"uller curves and prove the second statement of Theorem~\ref{thm:classification}. 
    \item In Section~\ref{sec:small}, we take care of small genera $g=5,6,7$, completing the proof of Theorem~\ref{thm:weights}. 
\end{itemize}

\paragraph{\bf AI usage disclosure}
The author used ChatGPT-6 to assist in developing proof ideas in Sections~\ref{sec:galois}--\ref{sec:small}, as part of an effort to generalize the results of \cite{Lee} to other strata including \(\hyp\).
The author verified the core mathematical  arguments and wrote the proofs.

\subsection{Cusps of Teichm\"uller curves}

Let \(C\subset \Omega\mathcal{M}_g(2g-2)^{\mathrm{hyp}}\) be an
algebraically primitive Teichm\"uller curve, with totally real trace field
\(K\) and \([K:\Q]=g\). Since there is only one zero, all the cusps of $C$ are irreducible. As in \cite[Section~6.1]{BHM}, we have the following description of the cusps of $C$.

\begin{proposition}\label{prop:MM}
Consider an irreducible horizontal cusp whose
normalization is a meromorphic \(1\)-form \(\omega_h^{\mathrm{norm}}\)
on \(\mathbb{P}^1_z\), written in two ways:
\[
\omega_h^{\mathrm{norm}}
=\sum_{i=1}^g c_i\Bigl(\frac{1}{z-x_i}-\frac{1}{z+x_i}\Bigr)\,dz
=\frac{R}{\prod_{i=1}^g (z^2-x^2_i)}\,dz,\]
where $R\in \R$ and the $2g$ node preimages $\pm x_i$ with $x_1=1$. By swapping $\pm x_i$ if necessary, we may assume $c_i>0$ and normalize to have $c_1=1$. The order $2g-2$ zero of \(\omega_h^{\mathrm{norm}}\) is at \(z=\infty\). Then
\begin{itemize}
    \item The set \(\{c_i\}\) forms a \(\Q\)-basis of \(K\).
    \item Let $(h_i)$ be the trace-dual basis of $(c_i)$. Then
    \[
       \frac{h_i c_k}{c_i h_k}\in\Q_{>0}\qquad1\le i,k\le g.
    \]
    For a generating differential $(X,\omega)$ near the cusp, the cylinder circumferences and heights are common positive real multiples of $(c_i)$ and $(h_i)$, respectively.
\end{itemize}
\end{proposition}

Furthermore, we set
\[
  P(Y)=\prod_{i=1}^g(Y-x^2_i).
\] The notation introduced above will be used throughout this paper.

Then we have the relation
\[
    \omega_h^{\mathrm{norm}}=\frac{2P'(1) dz}{P(z^2)}
\] and
\begin{equation*}
    c_i=\operatorname{res}_{x_i}\omega_h^{\mathrm{norm}} = \frac{P'(1)}{x_iP'(x^2_i)}.
\end{equation*}

The zero of order $2g-2$ at $\infty$ also gives

\[
\sum_{i=1}^g c_ix_i^{2r-1}=0,\qquad 1\le r\le g-1.
\]
Indeed, expanding the partial fractions at $z=\infty$ gives
\[
 \frac{\omega_h^{\mathrm{norm}}}{dz}
 =2\sum_{r\ge1}\left(\sum_{i=1}^g c_ix_i^{2r-1}\right)z^{-2r},
\]
and its coefficients of $z^{-2},\ldots,z^{-2g+2}$ vanish.

\section{Consequences of Harder--Narasimhan filtration and trace duality}\label{sec:HN-duality}

\subsection{Harder--Narasimhan filtration}
Let $f:\mathcal S\to\overline C$ be the semistable family of curves
(after a finite base change and compactification), and let $D$ denote the
section defined by the unique zero.  The extended Hodge bundle
\[
    \mathcal{E}=f_*\omega_{\mathcal S/\overline C}
\]
splits into eigenline bundles $\mathcal L_\sigma$ indexed by the real
embeddings $\sigma:K\hookrightarrow\R$.  The tautological line bundle is
denoted by $\mathcal L=\mathcal L_{\sigma_1}$, where $\sigma_1=1$.
Yu--Zuo identify the relevant Weierstrass filtration with the
Harder--Narasimhan filtration of $\mathcal{E}$
\cite[Proposition~5.5 and Theorem~5.6]{YZ}; the direct-sum principle in
\cite[Lemma~4.4]{BHM} identifies its steps with eigenlines.

\begin{proposition}
\label{prop:HN-minimal}
There is an ordering $\sigma_1=1,\sigma_2,\ldots,\sigma_g$
of the real embeddings of $K$ such that
\[
 0\subset V_1\subset\cdots\subset V_g=\mathcal{E},
 \qquad
 V_j=\bigoplus_{i=1}^j\mathcal L_{\sigma_i},\qquad
 \frac{\deg(V_j/V_{j-1})}{\deg\mathcal L}=1-\frac{2(j-1)}{2g-1}.
\]

Moreover,
\[
    V_j
    =
    f_*\bigl(
      \omega_{\mathcal S/\overline C}(-(2g-2j)D)
    \bigr).
\]
Thus $V_j$ consists of differentials vanishing to order at least
$2g-2j$ at the marked zero. A nonzero generator of
$\mathcal L_{\sigma_j}$ has exact order $2g-2j$ at the marked zero.
\end{proposition}

\begin{proof}
The unique zero is a Weierstrass point of the hyperelliptic curve.
The canonical vanishing sequence at a Weierstrass point is
\[
              0,2,4,\ldots,2g-2.
\]
Consequently
\[
 V_j
 =
 f_*\bigl(
   \omega_{\mathcal S/\overline C}(-(2g-2j)D)
 \bigr)
\]
has rank $j$. Its successive normalized slopes are given by Yu--Zuo
\cite[Proposition~5.5 and Theorem~5.6]{YZ}, and they are strictly decreasing. Since the Hodge bundle is a direct sum of
eigenlines, the direct-sum principle
\cite[Lemma~4.4]{BHM}, together with uniqueness of the
Harder--Narasimhan filtration, gives
\[
          V_j=\bigoplus_{i=1}^j\mathcal L_{\sigma_i}.
\]

The stable differentials compatible with hyperelliptic involution are
\[
               Q(z^2)\frac{dz}{P(z^2)},
               \qquad \deg Q\le g-1.
\]
If $\deg Q=m$, then at infinity
\[
       Q(z^2)\frac{dz}{P(z^2)}
          \sim z^{2m-2g}\,dz,
\]
and hence the sections vanishing to order at least $2g-2j$ are exactly
\[
      Q_{j}(z^2)\frac{dz}{P(z^2)},
      \qquad
      Q_j\in\C[Y],
      \qquad
      \deg Q_j\le j-1.
\]
This space has dimension \(j\), as on the smooth fibers.
Cohomology and base change therefore identifies it with
the fiber of \(V_j\) at the cusp. Finally, a nonzero vector in the fiber of
$\mathcal L_{\sigma_j}$ lies in $V_j$ but not in $V_{j-1}$.
Hence $Q_j$ has degree exactly $j-1$.
\end{proof}

The first consequences of the Harder--Narasimhan filtration are the following

\begin{proposition}\label{prop:deg}
The following statements hold:

\begin{enumerate}
    \item $\sum_{i=1}^g\sigma_j(c_i)x_i^{2r+1}=0,
 \qquad 0\le r\le g-j-1.$
    \item
For every $j$ there is a unique polynomial $Q_j\in\C[Y]$ of degree
exactly $j-1$ such that
\[
                    Q_j(x^2_i)=\frac{\sigma_j(c_i)}{c_i}.
\]
\end{enumerate}

\end{proposition}
\begin{proof}
As in Section~6 of \cite{BHM}, the eigenform associated with $\sigma_j$ at
this cusp can be normalized as
\[
 \omega_h^j
 =\sum_{i=1}^g\sigma_j(c_i)
 \left(\frac1{z-x_i}-\frac1{z+x_i}\right)dz.
\]
Note that the node preimages are unchanged. By Proposition~\ref{prop:HN-minimal}, the form $\omega_h^j$ belongs
to $V_j$ but not to $V_{j-1}$, with $V_0=0$. The description of $V_j$
in that proposition gives
\[
 \omega_h^j=Q_j(z^2)\omega_h^{\mathrm{norm}},
 \qquad \deg Q_j=j-1.
\]  $Q_j$ is uniquely determined by its values at the $g$ distinct points $x^2_i$. Taking residues at $x_i$ proves (2). Moreover, expanding $\omega^j_h$ at $z=\infty$, the order $2g-2j$ zero, gives (1).
\end{proof}

\subsection{Trace duality}

By Proposition~\ref{prop:MM}, there exist $\rho\in K^\times$ and
$q_i\in\Q_{>0}$ such that $h_i=\rho q_i c_i$. Since $c_1=1$, we may set $\rho=h_1$ and $h_i=h_1q_ic_i$.

Set \(D_i=q_i^{-1}\) and \(D=\diag(D_1,\ldots,D_g)\).
Let \(m=(m_1,\ldots,m_g)\) be the {\em primitive
moduli vector}, defined by the unique vector of positive
integers proportional to \((h_i/c_i)_i\) and satisfying
\(\gcd(m_1,\ldots,m_g)=1\). Since \(D_1=1\), we have
\[
    D_i=\frac{m_1}{m_i}.
\]
\begin{lemma}\label{lem:top-duality}
Let $S=(\sigma_j(c_i))_{i,j}=(c_iQ_j(x^2_i))_{i,j}$, with rows indexed by nodes
and columns indexed by embeddings, and let
$H=\diag(\sigma_1(\rho),\ldots,\sigma_g(\rho))$.  Then

\[
             SHS^{\mathsf T}=D, \qquad S^{\mathsf T}D^{-1}S=H^{-1}.
\]

In particular, $S^{\mathsf T}D^{-1}S$ is diagonal. As a result, we have

\begin{enumerate}
    \item $x_i=\sigma_g(D^{-1}_ic_i) = \sigma_g(h_i/h_1)$ and $K=\Q(x_1,\dots,x_g)$.
    \item $K/\Q$ is Galois.
\end{enumerate}
\end{lemma}

\begin{proof}
Trace duality gives
\[
            \Tr_{K/\Q}(\rho c_i c_j)=D_i\delta_{ij}, \qquad    SHS^{\mathsf T}=D.
\]
Since \(D\) is invertible, so is \(S\), and hence
\[
    S^{\mathsf T}D^{-1}S=H^{-1}.
\]

Taking \(r=0\) in (1) of Proposition~\ref{prop:deg}, we see that
\((x_i)_i\) is orthogonal, for the standard bilinear pairing,
to the first \(g-1\) columns of \(S\). The same is true of
the last column of \(D^{-1}S\). These vectors are therefore
proportional. Since \(x_1=c_1=D_1=1\), they are indeed the same, so
\[
    x_i=D_i^{-1}\sigma_g(c_i)
        =\sigma_g(h_i/h_1).
\]
In particular, \(x_i\in\sigma_g(K)\). On the other hand,
the residue formula
\[
    c_i=\frac{P'(1)}{x_iP'(x_i^2)}
\]
gives
\[
    K\subseteq\Q(x_1,\ldots,x_g)\subseteq\sigma_g(K).
\]
Both $K$ and $\sigma_g(K)$ have degree \(g\) over \(\Q\).
Thus $K=\Q(x_1,\ldots,x_g)=\sigma_g(K).$

It remains to show that every $\sigma_j$ preserves \(K\).
Let $w_i=\frac{c_i^2}{D_i}\in K_{>0},$
and consider the pairing on polynomials of degree at most \(g-1\)
given by
\[
    \langle A,B\rangle_*
    =\sum_{i=1}^g w_iA(x^2_i)B(x^2_i).
\]
This pairing is
positive definite on real polynomials of degree at most \(g-1\), since \(x^2_i\) are distinct and real. The corresponding matrix in the monomial basis has entries in \(K\). Gram--Schmidt applied to \(1,Y,\ldots,Y^{g-1}\) gives unique monic orthogonal polynomials \(R_0,\ldots,R_{g-1}\in K[Y]\) with \(\deg R_j=j\). On the other hand, the polynomials \(Q_1,\ldots,Q_g\) are also
orthogonal for this pairing since \(S^{\mathsf T}D^{-1}S=H^{-1}\). Since \(\deg Q_j=j-1\),
\(Q_j\) is proportional to \(R_{j-1}\). Moreover,
\[
    Q_j(1)=\frac{\sigma_j(c_1)}{c_1}=1.
\]
Hence \(R_{j-1}(1)\ne0\) and
\[
    Q_j(Y)=\frac{R_{j-1}(Y)}{R_{j-1}(1)}\in K[Y].
\]
It follows that
\[
    \sigma_j(c_i)=c_iQ_j(x_i^2)\in K
\]
for every \(i,j\). Since the \(c_i\) generate \(K\), every
embedding \(\sigma_j\) preserves \(K\). Thus \(K/\Q\) is Galois.
\end{proof}

\section{\texorpdfstring{The Galois group $\operatorname{Gal}(K/\Q)$}{The Galois group Gal(K/Q)}} \label{sec:galois}

Since we have proved $K/\Q$ is Galois, the embeddings $\sigma_1,\dots, \sigma_g$ are now elements of the Galois group $\operatorname{Gal}(K/\Q)$. In this section, we study how the Harder--Narasimhan filtration restricts the group structure of $\operatorname{Gal}(K/\Q)$. For convenience, we relabel the elements of $\operatorname{Gal}(K/\Q)$ as\[
    \tau_r \coloneqq \sigma_g^{-1}\sigma_{g-r},\qquad r=0,\dots, g-1.
\]

Set
\[
    u_i=\frac{c_i}{D_i}=\sigma_g^{-1}(x_i),
    \qquad
    e_r=(\tau_r(u_i))_{i=1}^g,
    \qquad 0\le r\le g-1.
\]
We use the positive definite pairing
\[
    \langle v,w\rangle_D=v^{\mathsf T}Dw.
\]
The vectors \(e_0,\ldots,e_{g-1}\) are orthogonal in this pairing by Lemma~\ref{lem:top-duality}.

\begin{lemma} \label{lem:odd}
For each \(0\le r\le g-1\), there exists a unique odd
polynomial \(F_r\in K[z]\) of degree \(2r+1\) such that
\[
    \tau_r(u_i)=F_r(u_i)
    \qquad (1\le i\le g).
\] Moreover, if $2rs+r+s<g$, then
\[
    \tau_r\tau_s =\tau_s\tau_r= \tau_{2rs+r+s},\qquad \tau_r(F_s)\circ F_r=\tau_s(F_r)\circ F_s = F_{2rs+r+s}
\]
\end{lemma}

\begin{proof}
Applying \(\sigma_g^{-1}\) to (1) of Proposition~\ref{prop:deg}, with
\(j=g-s\), gives
\[
    \sum_iD_i\tau_s(u_i)u_i^{2r+1}=0,
    \qquad 0\le r<s.
\]
Thus \((u_i^{2r+1})_i\) is orthogonal to
\(e_{r+1},\ldots,e_{g-1}\), and therefore belongs to
\(\Span_K(e_0,\ldots,e_r)\).

The numbers \(u_i\) are nonzero and their squares are distinct. So the Vandermonde determinant shows that
\[
    (u_i)_i,\ (u_i^3)_i,\ \ldots,\ (u_i^{2g-1})_i
\]
are linearly independent. By dimension comparison, we obtain
\[
    \Span_K\{(u_i^{2k+1})_i:0\le k\le r\}
    =\Span_K(e_0,\ldots,e_r).
\]
This gives the unique polynomial \(F_r\) of degree
\(2r+1\).

Suppose now that \(2rs+r+s<g\), and denote
\(\tau_r\tau_s=\tau_t\). Then
\[
    F_t(u_i)
    =\bigl(\tau_r(F_s)\circ F_r\bigr)(u_i).
\]
Both sides are evaluations of odd polynomials of degree
at most \(2g-1\). Their difference is \(zG(z^2)\), where
\(\deg G\le g-1\), and \(G\) vanishes at the \(g\) distinct
points \(u_i^2\). Hence $G\equiv 0$ and the two polynomials coincide.
The degree comparison gives
\[
    2t+1=(2r+1)(2s+1),
    \qquad t=2rs+r+s.
\]
\end{proof}

\subsection{\texorpdfstring{Tridiagonal matrix $J_0$}{Tridiagonal matrix J0}}

For each \(r\), choose \(\varepsilon_r\in\{\pm1\}\) so that
\(\varepsilon_rF_r\) has positive leading coefficient, and set
\[
    \phi_r=\varepsilon_r\frac{e_r}{\|e_r\|_D},
    \qquad
    U=(e_0\mid\cdots\mid e_{g-1}),
    \qquad
    \Phi=(\phi_0\mid\cdots\mid\phi_{g-1}).
\]
Then \(\Phi^{\mathsf T}D\Phi=I\), and
\(\Lambda=\Phi^{-1}U\) is an invertible diagonal matrix.
Define
\[
    \Delta_0=\diag(u_1^2,\ldots,u_g^2),
    \qquad
    J_0=U^{-1}\Delta_0U,
    \qquad
    \mathcal J_0=\Phi^{-1}\Delta_0\Phi.
\]
Thus \(J_0=\Lambda^{-1}\mathcal J_0\Lambda\). The rows and columns of \(J_0\) and \(\mathcal J_0\) are indexed by \(0,\ldots,g-1\). The matrices $\Delta_0,U,J_0$ have entries in $K$.

Since $\Delta_0$ is diagonal, we have
\[
    (\mathcal J_0)_{rs}=\langle \phi_r,\Delta_0\phi_s \rangle_D=\langle \Delta_0\phi_r,\phi_s \rangle_D =(\mathcal J_0)_{sr}.
\] That is, $\mathcal J_0$ is symmetric. On the other hand, multiplication by \(u_i^2\) increases the degree of an odd evaluation polynomial by two. Thus \(\Delta_0\phi_r\in\Span_\R(\phi_0,\ldots,\phi_{r+1})\)
for \(0\le r\le g-2\). By symmetry,
\((\mathcal J_0)_{rs}=0\) whenever \(|r-s|>1\). Thus $\mathcal J_0$ is tridiagonal. Moreover, its adjacent off-diagonal
entries are positive by comparison of leading coefficients of $u_i^2\phi_r$ and $\phi_{r+1}$.

The tridiagonal matrix $J_0$ encodes the group structure of $\operatorname{Gal}(K/\Q)$ as follows.

\begin{lemma}
\label{lem:jacobi-covariance}
Let \(\pi_r\) be the permutation of \(\{0,\ldots,g-1\}\)
characterized by
\[
    \tau_r\tau_j=\tau_{\pi_r(j)},
\]
and let \(P_r\) be its permutation matrix.
Denote \(F_r(z)=zR_r(z^2)\), with \(R_r\in K[Y]\) of degree \(r\).
Then
\begin{equation}\label{eq:general-covariance}
    \tau_r(J_0)
    =P_r^{-1}J_0R_r(J_0)^2P_r.
\end{equation}

Consequently, the off-diagonal support graph of
\(J_0R_r(J_0)^2\) is the image under \(\pi_r\) of the path
\[
    \mathcal P
    =\bigl\{\{j,j+1\}:0\le j<g-1\bigr\}.
\]
This matrix has nonzero entries at $(i,j)$ only if \(|i-j|\le 2r+1\), and every entry
at exactly \(|i-j|=2r+1\) is nonzero.
\end{lemma}

\begin{proof}
The definition of \(P_r\) gives $\tau_r(U)=UP_r$. Furthermore,
\[
    \tau_r(\Delta_0)
    =\diag\bigl(F_r(u_1)^2,\ldots,F_r(u_g)^2\bigr)
    =\Delta_0R_r(\Delta_0)^2.
\]
Applying \(\tau_r\) to \(J_0=U^{-1}\Delta_0U\) proves
\eqref{eq:general-covariance}. An automorphism of \(K\) preserves whether an entry is zero.
Thus \(\tau_r(J_0)\) has the same support graph \(\mathcal P\) as \(J_0\). Equation~\eqref{eq:general-covariance} gives the path permuted by $\tau_r$.

Finally, \(\deg(YR_r(Y)^2)=2r+1\). If \(\ell_r\) is the
leading coefficient of \(R_r\), then
\[
    \bigl(J_0R_r(J_0)^2\bigr)_{i,i+2r+1}
    =\ell_r^2
     \prod_{k=i}^{i+2r}(J_0)_{k,k+1}\ne0.
\]
\end{proof}

These lemmas indicate that we can recover the group structure of $\operatorname{Gal}(K/\Q)$, from the labeling $\tau_r$. At the end of this section, we prove

\begin{theorem}\label{thm:weights}
Suppose that \(g\ge5\) and that
\(\Omega\mathcal M_g(2g-2)^{\mathrm{hyp}}\) contains an
algebraically primitive Teichm\"uller curve. Then \(2g+1\) is prime or \(g\) is a power of two. Also, we have an isomorphism
    \[
        \deg : \operatorname{Gal}(K/\Q)\to (\Z/N\Z)^\times/\{\pm1\},\qquad \tau_r \mapsto \pm \deg F_r = \pm(2r+1) \pmod N
    \] for $N=2g+1$ prime, or $N=4g$ a power of two. At every cusp, the primitive moduli vector is, up to permutation,
\[
 (1,\ldots,1)\qquad\text{or}\qquad(2,1,\ldots,1).
\] The second moduli vector requires \(g\) to be a power of two.
\end{theorem}

\subsection{\texorpdfstring{Computation of $F_r$}{Computation of Fr}}

In this subsection, we assume $g\ge 8$. Let $T_n$ denote the Chebyshev polynomial satisfying $T_n(\cos \theta)=\cos(n\theta)$. For example, $T_2(z)=2z^2-1$ and $T_3(z)=4z^3-3z$. In particular, they satisfy $T_n\circ T_m=T_{mn}$. We now compute the odd polynomials \(F_1\) and \(F_2\) explicitly.
Write
\[
    F_1(z)=az^3+bz,\qquad
    F_2(z)=cz^5+dz^3+ez.
\]
Then \(a,c\ne0\), \(a+b=1\), and \(c+d+e=1\) by evaluating at $u_1=1$. On the other hand,
\[
    \langle e_0,e_1 \rangle_D = a\sum_i D_iu_i^4 + b \sum_i D_i u_i^2 =0,
\] so $b/a<0$. Set $s\coloneqq \sqrt{-3a/4b}$ so that
\[
    F_1(z)=\frac{T_3(sz)}{T_3(s)}.
\]

Since \(2g-1\ge 15\), Lemma~\ref{lem:odd} gives $\tau_1\tau_2=\tau_2\tau_1$ and $\tau_1(F_2)\circ F_1=\tau_2(F_1)\circ F_2.$ Comparing coefficients of degrees \(15,13,11,9\), in order,
gives
\[
 \tau_1(c)=\frac{\tau_2(a)c^3}{a^5},\quad
 d=\frac{5bc}{3a},\quad
 e=\frac{5b^2c}{9a^2},\quad
 \tau_1(d)=\frac{5\tau_2(a)b^3c^3}{27a^6}.
\]

By direct computation we obtain
\[
    F_2(z)=\frac{T_5(sz)}{T_5(s)}
\]

Moreover, we obtain
\[
    \frac{5\tau_2(a)b^3c^3}{27a^6}=\tau_1(d) = \frac{5\tau_1(b)}{3\tau_1(a)}\tau_1(c) = \frac{5\tau_1(b)}{3\tau_1(a)}\frac{\tau_2(a)c^3}{a^5}.
\] In particular,

\[
\tau_1(b/a)=\frac{b^3}{9a},\qquad \tau_1(s^2)=T_3(s)^2.
\]

Denote $v_i=su_i$. Since $s^2\in K$, we can extend \(\tau_1\) to an automorphism
\(\widetilde\tau_1\) of \(K(s)\), with
\(\widetilde\tau_1(s)=\varepsilon T_3(s)\) for some $\varepsilon=\pm1$. Since \(T_3\) is odd and has rational coefficients,

\[
     \widetilde\tau_1(v_i)=\widetilde\tau_1(s)\tau_1(u_i)=\varepsilon T_3(s)F_1(u_i)= \varepsilon T_3(v_i).
\] Also from $\tau_1^2=\tau_4$, we have  $\widetilde\tau^2_1(s)\tau_4(u_i)=\widetilde\tau^2_1(v_i) = \varepsilon^2 T_3\circ T_3 (v_i)=T_9(v_i)$. So vectors $f_r\coloneqq (T_{2r+1}(v_i))_i$ are orthogonal with the metric $D$, as they are proportional to $e_r$, for $r=0,1,2,4$.
Let
\[
 \mathcal N\coloneqq \|f_1\|^2_D = \langle f_1,f_1\rangle_D ,\qquad
 C_k=\sum_iD_iT_k(v_i).
\]
The product-to-sum identity \(2T_rT_t=T_{r+t}+T_{|r-t|}\) among Chebyshev polynomials, applied to the orthogonal pairs \((1,3)\), \((1,5)\), \((3,5)\), \((1,9)\), \((3,9)\), and \((5,9)\), gives
\[
 C_4=-C_2,\quad C_6=C_2,\quad C_8=-C_2,\quad
 C_{10}=C_2,\quad C_{12}=-C_2,\quad C_{14}=C_2.
\]
Thus \[\langle f_3,f_0\rangle_D = \sum_i D_i T_{1}(v_i)T_{7}(v_i)= \frac{1}{2}\sum_i D_i (T_{6} (v_i) + T_{8} (v_i)) =\frac{C_{6}+C_{8}}{2}=0.\] Similarly, $\langle f_3,f_1\rangle_D=\langle f_3,f_2\rangle_D=0$ and $f_3$ is orthogonal to $f_0,f_1,f_2$. By degree comparison, it is also contained in the orthogonal complement of $f_4$.

By Lemma~\ref{lem:odd}, \(f_3\in\Span(e_0,e_1,e_2,e_3)\). Since \(f_3\) is orthogonal to \(f_0,f_1,f_2\), it is a nonzero multiple of \(e_3\).

For \(0\le r\le3\), the product-to-sum identity gives
\[
    \|f_r\|_D^2
    =\frac{C_0+C_{4r+2}}2
    =\frac{C_0+C_6}2
    =\mathcal N.
\]
In particular,
\[
    \mathcal N=\sum_iD_iv_i^2\in K.
\]
Since \(\widetilde\tau_1(v_i)=\varepsilon T_3(v_i)\),
we have
\[
    \tau_1(\mathcal N)=\|f_1\|_D^2=\mathcal N.
\]
Using \(\tau_4=\tau_1^2\), we similarly obtain
\[
    \|f_4\|_D^2=\tau_4(\mathcal N)=\mathcal N.
\]
This gives
\[
    \phi_r=\frac{f_r}{\sqrt{\mathcal N}},
    \qquad 0\le r\le4.
\]
For the remaining indices, we retain the orthonormal vectors
\(\phi_r\) defined before Lemma~\ref{lem:jacobi-covariance}.

\subsection{Computation of the tridiagonal matrix}

Let \(\Delta \coloneqq 2s^2\Delta_0-I= \operatorname{diag} (T_2(v_i))=\operatorname{diag} (2v_i^2-1)\).

Consider the symmetric matrix

\[
    J\coloneqq \Phi^{-1} \Delta \Phi = 2s^2 \mathcal{J}_0-I.
\]

$J$ is tridiagonal and we set
\[
    \Delta \phi_r =  \beta_{r-1} \phi_{r-1} + \alpha_r \phi_r + \beta_r \phi_{r+1}.
\] Moreover, we always have $\beta_r>0$ by the leading coefficient comparison.

\[
J=
\begin{pmatrix}
\alpha_0&\beta_0&0&\cdots&0\\
\beta_0&\alpha_1&\beta_1&\ddots&\vdots\\
0&\beta_1&\alpha_2&\ddots&0\\
\vdots&\ddots&\ddots&\ddots&\beta_{g-2}\\
0&\cdots&0&\beta_{g-2}&\alpha_{g-1}
\end{pmatrix}.
\]

From computation with $r=0,1,2,3$, we have
\[
\alpha_0=\tfrac12,\quad \alpha_1=\alpha_2=\alpha_3=0,
 \qquad \beta_0=\beta_1=\beta_2=\beta_3=\tfrac12.
\]

Set \(B=U^{-1}\Delta U\). Since
\(\Lambda=\Phi^{-1}U\) is invertible and diagonal, we have
\[
    B=\Lambda^{-1}J\Lambda,
    \qquad
    T_3(B)=\Lambda^{-1}T_3(J)\Lambda.
\]
Thus \(T_3(B)\) and \(T_3(J)\) have the same zero pattern.
Also,
\[
    B=2s^2J_0-I\in M_g(K).
\]

The identity \(T_3\circ T_2=T_2\circ T_3\), together with
\(\widetilde\tau_1(v_i)=\varepsilon T_3(v_i)\), gives
\[
    \tau_1(\Delta)=T_3(\Delta).
\]
Let \(P_1\) be the permutation matrix defined by
\(\tau_1(U)=UP_1\). Applying \(\tau_1\) to
\(B=U^{-1}\Delta U\), we obtain
\[
    \tau_1(B)
    =P_1^{-1}U^{-1}\tau_1(\Delta)UP_1
    =P_1^{-1}T_3(B)P_1.
\]

The off-diagonal support graph of \(B\), and hence of
\(\tau_1(B)\), is the path
\[
    \mathcal P
    =\bigl\{\{j,j+1\}:0\le j<g-1\bigr\}.
\]
By Lemma~\ref{lem:jacobi-covariance}, the off-diagonal support graph of
\(T_3(B)\), and hence of \(T_3(J)\), is the
\(\mathcal P\) permuted by \(\tau_1\).
It has exactly \(g-1\) edges, or equivalently \(g-1\)
nonzero entries strictly above the diagonal.

Since \(J\) is tridiagonal, \(T_3(J)=4J^3-3J\) has nonzero entries at $(i,j)$ only if $|i-j|\le3$. Since \(\beta_r>0\), we have
\[
    (T_3(J))_{r,r+3}
    =4\beta_r\beta_{r+1}\beta_{r+2}>0,
    \qquad 0\le r\le g-4.
\] We already have $g-3$ nonzero entries at $(0,3),\dots, (g-4,g-1)$. The indices $3,4,\dots,g-4$ appeared twice and cannot appear anymore. Therefore
\[
     (T_3(J))_{i,i+2}=0\ (1\le i\le g-4),\qquad
     (T_3(J))_{i,i+1}=0\ (2\le i\le g-4).
\]

Direct multiplication gives
\[
 (T_3(J))_{i,i+2}
 =4\beta_i\beta_{i+1}(\alpha_i+\alpha_{i+1}+\alpha_{i+2}),
 \]
 \[
 (T_3(J))_{i,i+1}
 =\beta_i\bigl[4(\beta_{i-1}^2+\beta_i^2+\beta_{i+1}^2
 +\alpha_i^2+\alpha_i\alpha_{i+1}+\alpha_{i+1}^2)-3\bigr].
\]

This implies $\alpha_1=\dots=\alpha_{g-2}=0$ and $\beta_0=\dots=\beta_{g-3}=1/2$. So $(T_3(J))_{0,2} = 4\beta_0\beta_1\alpha_0=\frac{1}{2}$ and the index $0$ also appeared twice. The index $2$ cannot appear anymore because all the candidates are already dropped.

Now it remains to determine $\alpha\coloneqq \alpha_{g-1}$ and $\beta\coloneqq \beta_{g-2}>0$. The undetermined part of $T_3(J)$ looks like

\[
\begin{pmatrix}
0&0&0&\beta\\
0&0&A&C\\
0&A&4\alpha\beta^2&B\\
\beta&C&B&\alpha(4\alpha^2+8\beta^2-3)
\end{pmatrix},
\]
where
\[
\begin{aligned}
    A&=T_3(J)_{g-3,g-2}=\frac{4\beta^2-1}{2},\\
    B&=T_3(J)_{g-2,g-1}=\beta(4\beta^2+4\alpha^2-2),\\
    C&=T_3(J)_{g-3,g-1}=2\alpha\beta.
\end{aligned}
\]

Exactly one of the three entries is nonzero. Setting the other two equal to zero yields
\[
 (\alpha,\beta^2)\in
 \{(\tfrac12,\tfrac14),(-\tfrac12,\tfrac14),
   (0,\tfrac14),(0,\tfrac12)\}.
\]

We call these Cases A, B, C, and D, respectively. For each case, up to scalar multiplication, the characteristic polynomial of $J$ is

\[
 (y-1)U_{g-1}(y),\qquad T_g(y),\qquad
 U_g(y)-U_{g-1}(y),\qquad T_g(y)-T_{g-1}(y),
\] respectively, where $U_n$ is the Chebyshev polynomial defined by $U_n(\cos\theta)=\frac{\sin((n+1)\theta)}{\sin\theta}$.

Since the eigenvalues of \(J\) are \(T_2(v_i)\) and \(v_i>0\),
the four cases give the following values, up to
permutation:
\begin{center}
\begin{tabular}{@{}cccc@{}}
\toprule
Case & \(\alpha\) & \(\beta^2\) & \(v_i\), indexed by \(0\le k<g\)\\
\midrule
A & \(1/2\) & \(1/4\) & \(\cos(k\pi/(2g))\)\\[2pt]
B & \(-1/2\) & \(1/4\) & \(\cos((2k+1)\pi/(4g))\)\\[2pt]
C & \(0\) & \(1/4\) & \(\cos((2k+1)\pi/(4g+2))\)\\[2pt]
D & \(0\) & \(1/2\) & \(\cos(k\pi/(2g-1))\)\\
\bottomrule
\end{tabular}
\end{center}

Let \(L=\Q(v_1,\ldots,v_g)\). Since \(u_1=1\) and $s^2\in K$,

\[
 K=\Q(u_i:1\le i\le g),\qquad
 L=K(s),\qquad [L:K]\le2.
\]

\begin{proof}[Proof of Theorem~\ref{thm:weights} for $g\ge8$]

In Case A, the value at $k=0$ is \(1\), so \(s\in K\) and \(K=L=\Q(\cos(\pi/(2g)))=\Q(\zeta_{4g})^+\). Thus \(\varphi(4g)=2g\), which forces \(g\) to be a power of two. In Case B, \(L=\Q(\cos(\pi/(4g)))=\Q(\zeta_{8g})^+\). The cyclotomic
automorphism \(\zeta_{8g}\mapsto\zeta_{8g}^{4g+1}\) sends $\cos((2k+1)\pi/(4g))\mapsto -\cos((2k+1)\pi/(4g))$, so it is a nontrivial automorphism that fixes all $u_i$. Hence \([L:K]=2\), giving \(\varphi(8g)=4g\), again forcing \(g\) to be a power of two. In Case C, let \(N=2g+1\). Then \(L=\Q(\cos(\pi/(2N)))\) has degree \(\varphi(N)\). The automorphism \(\zeta_{4N}\mapsto\zeta_{4N}^{2N+1}\)
sends \(v_i\mapsto -v_i\), so \([L:K]=2\). Therefore
\(\varphi(N)=2g=N-1\), and \(N=2g+1\) is prime. Case D is impossible: \(L=\Q(\cos(\pi/(2g-1)))=\Q(\zeta_{2g-1})^+\) and $[K:\Q]\le [L:\Q]\le g-1$. This proves the first part of Theorem~\ref{thm:weights}.

It remains to compute $D_i$. In the possible cases A--C, all
adjacent off-diagonal entries of \(J\) equal \(1/2\), the recurrence relation
\[
    (2x^2-1)T_{n}(x)=\frac{1}{2}\left[T_{n-2}(x)+T_{n+2}(x)\right]
\]  determines $F_r(z)=T_{2r+1}(sz)/T_{2r+1}(s)$ and $\|f_r\|^2_D=\mathcal N$ for each $r$. Let \(v_i=\cos\theta_i\). Then square matrix
\(\bigl(\sqrt{D_i/\mathcal N}\cos((2r+1)\theta_i)\bigr)_{i,r}\) is orthonormal in the standard metric, so
\[
 D_i\sum_{r=0}^{g-1}\cos^2((2r+1)\theta_i)=\mathcal N
\] for each $i$.

The finite sum in the formula is
\[
 \begin{cases}
 g,&\text{case A with }\theta_i=0,\\
 g/2,&\text{case A with }\theta_i\ne0,\text{ or case B},\\
 (2g+1)/4,&\text{case C}.
 \end{cases}
\]

Consequently the \(D_i\) are equal in Case B and Case C.
In Case A, $D_i$ at \(v_i=1\) is half every other weight. Since \(D_i^{-1}=m_i/m_1\), this proves the assertion about
the primitive moduli vector.

It remains to identify the degree labeling with the
cyclotomic Galois action. In Cases A and B, set \(N=4g\);
in Case C, set \(N=2g+1\). The preceding field calculations
give
\[
    K=\Q(\zeta_N)^+.
\]
Indeed, in Cases B and C the automorphism sending every
\(v_i\) to \(-v_i\) has fixed field \(\Q(\zeta_N)^+\);
this field has degree \(g\), so it equals \(K\).

For \(n=2r+1\), let \(\gamma_n\) be the cyclotomic
automorphism of \(L\) induced by raising roots of unity
to the \(n\)-th power. The displayed cosine formulas give
\[
    \gamma_n(v_i)=T_n(v_i),
    \qquad
    \gamma_n(s)=T_n(s).
\]
Therefore
\[
    \gamma_n(u_i)
    =\frac{T_n(v_i)}{T_n(s)}
    =F_r(u_i)
    =\tau_r(u_i).
\]
Since the \(u_i\) generate \(K\), the restriction of
\(\gamma_n\) to \(K\) is \(\tau_r\). The odd integers
\(1,3,\ldots,2g-1\) represent all classes in
\((\Z/N\Z)^\times/\{\pm1\}\). This proves the claimed
isomorphism
\[
    \tau_r\longmapsto\pm(2r+1)\pmod N.
\]
\end{proof}

The cases \(g=5,6, 7\) are treated in
Section~\ref{sec:small}.

\section{Orbifold points and elliptic maps}\label{sec:elliptic}

In this section, we investigate the orbifold points of the algebraically primitive Teichm\"uller curves and elliptic maps in their Veech group. This will lead to the complete classification. In particular, we prove

\begin{proposition}\label{prop:elliptic-generator}
Let \(C\) be algebraically primitive in \(\hyp\), with \(g\ge5\).
If the primitive moduli vector at every cusp satisfies
\[
    \sum_{i=1}^g m_i\le g+1
\]
then the Veech group contains an elliptic element \(E\) such that
\[
    K=\Q(\operatorname{tr}(E)).
\]
\end{proposition}

\begin{proof}

Let \(\Gamma\subset\SL_2(\R)\) be the Veech group and let
\(\overline\Gamma\subset\operatorname{PSL}_2(\R)\) be its
projective image. We regard
\[
    C=\HH/\overline\Gamma
\]
as an effective orbifold. Let \(h\) be the genus of its
smooth compactification, \(c\) its number of cusps, and
\(e_1,\ldots,e_r\) its finite orbifold orders. Then
\begin{equation}\label{eq:area}
    \chi_{\mathrm{orb}}(C)
    =2-2h-c-\sum_{\nu=1}^r(1-e_\nu^{-1})
    =-\frac{\operatorname{Area}(\HH/\overline\Gamma)}{2\pi}
    <0.
\end{equation}

Since every Teichm\"uller curve has at least one cusp obtained by stretching a periodic direction, we have $c\ge 1$.

At cusp \(y_j\), denote the horizontal cylinders as
\([0,c_{ji}]\times[0,h_{ji}]\). A shear
\((x,y)\mapsto(x+sy,y)\) twists the \(i\)-th cylinder
\(s h_{ji}/c_{ji}\) times. The primitive moduli vector
\((m_{j1},\ldots,m_{jg})\) is proportional to
\((h_{ji}/c_{ji})_i\). It determines the smallest positive
shear that is an integral twist in every cylinder. The resulting affine automorphism \(T_j\) fixes all
cylinder boundaries pointwise and twists the \(i\)-th cylinder
\(m_{ji}\) times.

Let \(P_j\) be the primitive parabolic element that generates the cusp $y_j$ with positive trace. There is an integer \(k_j\ge1\) such that
\(T_j= P_j^{\,k_j}\).

By Selberg's theorem $\Gamma$ contains a finite index torsion free group. Also by taking further subgroup $\Gamma'$ if necessary, we take a torsion-free cover \(C'=\HH/\Gamma'\to C\) of degree \(d\) with its
ramification index above \(y_j\) \(a\) divisible
by \(k_j\).

The local monodromy at $y_j$ is then
\[
P_j^a=T_j^{a/k_j}.
\] Thus the local smoothing parameter of the \(i\)-th node
has vanishing order \(a m_{ji}/k_j\). Since there are \(d/a\) points above
\(y_j\), its contribution to the pullback of $\delta_{\mathrm{irr}}$ is
\[
\deg \delta_j
=\frac da\sum_i\frac{a m_{ji}}{k_j}
=\frac d{k_j}\sum_i m_{ji}.
\]

Let \(\overline{C'}\) be the smooth compactification of \(C'\),
and let \(\partial C'=\overline{C'}\setminus C'\) be its
reduced cusp divisor. The maximal Higgs property gives
\[
    \deg\mathcal L
    =\frac12\deg\omega_{\overline{C'}}(\partial C')
    =-\frac12\chi(C')
    =-\frac d2\chi_{\mathrm{orb}}(C).
\]
Summing the normalized degrees in
Proposition~\ref{prop:HN-minimal}, we obtain
\[
    \deg\det\mathcal E
    =\frac{g^2}{2g-1}\deg\mathcal L
    =-\frac{dg^2}{2(2g-1)}\chi_{\mathrm{orb}}(C).
\]

Every cusp fiber is irreducible and all its nodes are fixed by
\(\iota\), with branches exchanged. Thus only the
\(\Xi_{\mathrm{irr}}\) boundary type occurs. The class formula on $\operatorname{Pic}(\overline{\mathcal H}_g)\otimes\Q$ from \cite[equation~(1)]{Cornalba} gives $(8g+4)(\deg \det \mathcal{E})=g\deg\delta_{\mathrm{irr}}$, as the Teichm\"uller curves in minimal strata only have irreducible boundaries (cusps). Consequently, we have
\[
 -\chi_{\mathrm{orb}}(C)=\frac{2g-1}{g(4g+2)}
      \sum_{j=1}^c\frac{\sum_i m_{ji}}{k_j}.
\]

Under the stated hypothesis,
\begin{equation}\label{eq:area-bound}
 -\chi_{\mathrm{orb}}(C)\le\frac{(2g-1)(g+1)}{g(4g+2)}c
   =\left(\frac12-\frac1{2g(2g+1)}\right)c<\frac c2.
\end{equation}

Combining \eqref{eq:area} and \eqref{eq:area-bound} forces
\(h=0\) and \(c\le3\). If \(c=3\), the sum of the finite
orbifold contributions is less than \(1/2\), so there are none.
If \(c=2\), that sum is strictly between zero and one, so there
is exactly one. If \(c=1\), it is strictly between one and
\(3/2\), so there are exactly two, of orders \(a\le b'\)
with \(1/a+1/b'>1/2\). The possible signatures are therefore
\[
 (\infty,\infty,\infty),\quad(n,\infty,\infty),\quad
 (2,n,\infty),\quad(3,3,\infty),\quad
 (3,4,\infty),\quad(3,5,\infty).
\]
All are triangle orbifolds. The ideal triangle and the last
three signatures have trace fields of degree at most two.

Since \([K:\Q]=g>2\), only \((2,n,\infty)\) and
\((n,\infty,\infty)\) remain. In either case the elliptic
generator of projective order \(n\) generates the trace field as $\operatorname{tr}(E)=2\cos(\pi/n)$ and

\[
    K=\Q\bigl(2\cos(\pi/n)\bigr),
    \qquad
    g=[K:\Q]=\frac{\varphi(2n)}2.
\]

\end{proof}

\begin{proposition} \label{prop:cusptopol}
    Let $C$ be an algebraically primitive Teichm\"uller curve in $\hyp$. Suppose that there exists an elliptic element $E$ of the Veech group such that $K=\Q(\operatorname{tr}(E))$. Then the curve is either the double regular $(2g+1)$-gon or the regular $4g$-gon.
\end{proposition}
\begin{proof}
Let \(n\) be the order of the projective image of \(E\).
At its fixed point in the Teichm\"uller disk, the
corresponding affine automorphism is a holomorphic
automorphism \(f:X\to X\), and
\[
    f^*\omega=\mu\omega
\]
for a root of unity \(\mu\). Since the trace of \(E\)
generates \(K\),
\[
    K=\Q(\zeta_{2n})^+,
    \qquad
    2g=\varphi(2n).
\]

Let \(\iota\) be the hyperelliptic involution and
\(\pi:X\to\PP^1\) the quotient map. The quadratic differential
\(q\) defined by \(\pi^*q=\omega^2\) belongs to
\[
    \mathcal Q_0(2g-3,-1^{2g+1}).
\]
The automorphism \(f\) descends to an automorphism
\(\bar f\) of \(\PP^1\) fixing the unique zero of \(q\).

Let \(m=\operatorname{ord}(\bar f)\).
Apart from the marked branch point, at most one branch
point is fixed by \(\bar f\); every other orbit has
length \(m\). Consequently,
\[
    m\mid 2g
    \qquad\text{or}\qquad
    m\mid 2g+1.
\]
If \(\eta\) is the local multiplier of \(\bar f\) at the
marked point, then \(\eta\) has order \(m\). Since \(q\)
has order \(2g-3\) there and \(\bar f^*q=\mu^2q\), we have
\[
    \mu^2=\eta^{2g-1}.
\]
The projective derivative order is therefore
\[
    n=\operatorname{ord}(\mu^2)
     =\frac{m}{\gcd(m,2g-1)}.
\]
Both \(2g\) and \(2g+1\) are coprime to \(2g-1\).
Hence \(n=m\), and in particular \(n\le2g+1\).

If \(n\) is odd, then
\[
    2g=\varphi(2n)=\varphi(n)\le n-1\le2g.
\]
Thus \(n=2g+1\) is prime, and the other \(2g+1\) branch
points form a single orbit. After a change of coordinate
and rescaling the differential, this gives
\[
    X:\ y^2=x^{2g+1}-1,
    \qquad
    \omega=\frac{dx}{y}.
\]

If \(n\) is even, the divisibility above gives \(n\mid2g\),
and
\[
    2g=\varphi(2n)\le n\le2g.
\]
Hence \(n=2g\) is a power of two. There is one further
fixed branch point, and the remaining \(2g\) branch
points form a single orbit. The resulting model is
\[
    X:\ y^2=x(x^{2g}-1),
    \qquad
    \omega=\frac{dx}{y}.
\]

These are the models of the double regular \((2g+1)\)-gon
and the regular \(4g\)-gon, respectively
\cite{Veech89,Veech92}.
\end{proof}

\begin{proof}[Proof of Theorem~\ref{thm:classification}]
For \(g\ge8\), Theorem~\ref{thm:weights} gives the arithmetic
restriction and the bound
\[
    \sum_i m_i\le g+1
\]
at every cusp. Propositions~\ref{prop:elliptic-generator}
and~\ref{prop:cusptopol} therefore give the classification.

For \(g=5,6\), Proposition~\ref{prop:five-six-weights} gives
\(m=(1,\ldots,1)\) at every cusp, so the same two propositions
give the classification in these genera.
Proposition~\ref{prop:seven} excludes genus seven.

Conversely, the indicated polygon surfaces are Veech
surfaces \cite{Veech89,Veech92}. Their trace fields are
\[
    \Q(\zeta_{2g+1})^+
    \quad\text{and}\quad
    \Q(\zeta_{4g})^+,
\]
respectively. These fields have degree \(g\) when
\(2g+1\) is prime and when \(g\) is a power of two,
respectively. Thus all the listed curves are
algebraically primitive.

In genera three and four, existence is supplied by the
double regular \(7\)-gon and the regular \(16\)-gon,
respectively. Both genera satisfy the stated arithmetic
alternative. This proves the existence statement for
every \(g>2\), and the classification statement for
every \(g\ge5\).
\end{proof}

\section{Genera five, six and seven}\label{sec:small}

In this section, we finish our proof of Theorem~\ref{thm:weights} for $g=5,6,7$ cases. 

\subsection{Non-existence in genus seven}

\begin{proposition}\label{prop:seven}
There is no algebraically primitive Teichm\"uller curve in
\(\Omega\mathcal M_7(12)^{\mathrm{hyp}}\).
\end{proposition}

\begin{proof}
Let \(\pi_1\) be the permutation defined by
\[
    \tau_1\tau_j=\tau_{\pi_1(j)}.
\]
Since \(\tau_1\ne1\), \(\pi_1\) has no fixed point.

Set \(F_1(z)=zR_1(z^2)\). By Lemma~\ref{lem:jacobi-covariance}, the off-diagonal
support graph of \(J_0R_1(J_0)^2\) is \(\pi_1\mathcal P\). The distance of each edge is at most three, and every pair of distance three is an edge. Thus the image path contains the blocks
\[
    0-3-6,\qquad 1-4,\qquad 2-5.
\]

Lemma~\ref{lem:odd} gives \(\tau_1^2=\tau_4\). Consequently,
\[
    \pi_1(0)=1,\qquad \pi_1(1)=4.
\]
The image list therefore begins with \((1,4)\).
The block \(0-3-6\) must remain consecutive, because
the two compulsory edges already give vertex \(3\)
degree two. The compulsory edge \(\{2,5\}\) also makes
\(2\) and \(5\) consecutive.

Concatenating these blocks, with every successive
difference at most three, leaves precisely three possibilities
\[
    (1,4,6,3,0,2,5),\qquad
    (1,4,2,5,6,3,0),\qquad
    (1,4,5,2,0,3,6).
\]
These permutations fix positions \(3,2,6\), respectively. This is a contradiction.
\end{proof}

\subsection{Genera five and six}

In $g=5,6$ we determine the $\tau_1$ and $\tau_2$ image paths, then set the matrix entries outside
those paths equal to zero. These equations force the recurrence of Case C, so the weight
calculations have already been done.

\begin{lemma}\label{lem:small-order}
The automorphisms $\tau_1$ and $\tau_2$ act on the ordered
embedding positions by the following permutations,
written as image lists:
\[
\begin{array}{c|c|c}
g&\tau_1&\tau_2\\ \hline
5&(1,4,3,0,2)&(2,3,1,4,0)\\
6&(1,4,5,2,0,3)&(2,5,0,4,3,1).
\end{array}
\]
\end{lemma}

\begin{proof}

Lemma~\ref{lem:odd} gives \(\tau_1^2=\tau_4\) and $\tau_1\tau_0=\tau_1$. The $\tau_1$-path therefore begins with \((1,4)\). For \(g=5\), the group \(\operatorname{Gal}(K/\Q)\) is cyclic of order five. Since $\tau_1^2=\tau_4$, the two possible labelings of $(\tau_0,\tau_1,\tau_2,\tau_3,\tau_4)$ are
\[
    (1,\tau_1,\tau_1^3,\tau_1^4,\tau_1^2),
    \qquad
    (1,\tau_1,\tau_1^4,\tau_1^3,\tau_1^2).
\]

In the second ordering, the $\tau_1$-path contains \(\{0,4\}\), contradicting Lemma~\ref{lem:jacobi-covariance}. The first ordering gives the two paths in the table.

For \(g=6\), the $\tau_1$-path contains the blocks \((0,3)\), \((1,4)\), and \((2,5)\). By Lemma~\ref{lem:jacobi-covariance}, there are four possible permuted paths:
\[
    (1,4,3,0,2,5),\quad
    (1,4,2,5,3,0),\quad
    (1,4,5,2,0,3),\quad
    (1,4,5,2,3,0).
\]
The first two fix positions \(5\) and \(2\), respectively, so neither is possible.

For the last path, \(\tau_1\) has order six and \(\tau_2=\tau_1^4\). The $\tau_2$-path is then
\[
    (2,5,4,1,0,3),
\] which does not contain \(\{0,5\}\). This contradicts Lemma~\ref{lem:jacobi-covariance}.

Thus the $\tau_1$-path is
\[
    (1,4,5,2,0,3),
\]
and \(\tau_1\) has order three. The ordering $(\tau_0,\tau_1,\tau_2,\tau_3,\tau_4,\tau_5)$ of the group elements is now
\[
    (1,\tau_1,\tau_2,\tau_1^2\tau_2,\tau_1^2,\tau_1\tau_2).
\]
If \(\operatorname{Gal}(K/\Q)\) were non-abelian of order six, it would be isomorphic to \(S_3\). Then \(\tau_2^2=1\) and \(\tau_2\tau_1=\tau_1^{-1}\tau_2\), giving the $\tau_2$-path
\[
    (2,3,0,1,5,4),
\] which also does not contain \(\{0,5\}\), so \(\operatorname{Gal}(K/\Q)\) is cyclic.

Since \(\tau_2\notin\langle\tau_1\rangle\), we have
\(\tau_2^2\in\{1,\tau_1,\tau_1^2\}\). Since $\operatorname{Gal}(K/\Q)$ is abelian, these
three possibilities give, respectively,
\[
    (2,5,0,4,3,1),\qquad
    (2,5,1,0,3,4),\qquad
    (2,5,4,1,3,0).
\] Only the first path contains \(\{0,5\}\). This proves the table.
\end{proof}

We next translate these path restrictions into equations
for the recurrence coefficients. With the notation of
Lemma~\ref{lem:jacobi-covariance}, denote $t_0=(J_0)_{00}$. Since \(\phi_0=(u_i)_i/\|(u_i)_i\|_D\), we have
\[
    t_0\coloneqq(J_0)_{00}
    =(\mathcal J_0)_{00}
    =\frac{\sum_iD_iu_i^4}{\sum_iD_iu_i^2}>0.
\]
We also denote
\[
    A\coloneqq \frac{\mathcal J_0}{t_0}.
\]
Then \(A\) is symmetric and tridiagonal, with positive
adjacent off-diagonal entries. Denote
\[
    A_{ii}=a_i,\qquad
    A_{i,i+1}=\sqrt{b_i},\qquad
    a_0=1,\quad b_i>0.
\]
Let \(p_r\) be the monic orthogonal polynomial of degree
\(r\) given by
\[
    \tau_r(u_i)
    =u_i\frac{p_r(u_i^2/t_0)}{p_r(1/t_0)},
    \qquad 0\le r\le g-1.
\]

Since \(J_0/t_0=\Lambda^{-1}A\Lambda\), diagonal similarity
and Lemma~\ref{lem:jacobi-covariance} show that the
off-diagonal support graph of \(Ap_r(A)^2\) is
\(\pi_r\mathcal P\).

In particular, the following matrices have the two off-diagonal
support graphs listed in Lemma~\ref{lem:small-order}:
\[
\begin{aligned}
    B&=Ap_1(A)^2=A(A-I)^2,\\
    C&=Ap_2(A)^2
      =A\bigl((A-I)(A-a_1I)-b_0I\bigr)^2.
\end{aligned}
\]

Direct multiplication gives
\begin{align}
 B_{i,i+2}&=(a_i+a_{i+1}+a_{i+2}-2)\sqrt{b_i b_{i+1}},\label{eq:small-B2}\\
 B_{i,i+1}&=(b_{i-1}+b_i+b_{i+1}
  +a_i^2+a_ia_{i+1}+a_{i+1}^2
  -2(a_i+a_{i+1})+1)\sqrt{b_i},\label{eq:small-B1}
\end{align}
where the boundary terms \(b_{-1},b_{g-1}\) are set to zero.

\begin{proposition}\label{prop:five-six-weights}
For \(g=5\) or \(6\), every cusp has the primitive moduli vector \(m=(1,\ldots,1)\).
\end{proposition}

\begin{proof}
\emph{Genus five.}
The $\tau_1$-path gives the nonzero entries of $B$:
\((1,4),(3,4),(0,3),(0,2)\). The vanishing entries, using
\eqref{eq:small-B2}--\eqref{eq:small-B1}, give
\begin{gather*}
 a_1=a_4=t,\qquad a_2+a_3=2-t,\qquad b_0=b_3=v,\\
 b_1=t-t^2-v,\qquad
 b_2=-a_3^2-a_3t+2a_3+t-1.
\end{gather*}

The $\tau_2$-path does not contain \((0,3),(0,1),(2,4)\), giving vanishing entries of
\(C\), respectively,
\begin{align*}
 0&=1-t-b_1-2v,\\
 0&=b_1\bigl((a_3-1)(t-1)-v\bigr),\\
 0&=(1-a_3)b_1+(a_3+3t-3)v.
\end{align*}
The first equation gives \(v=(1-t)^2>0\), so \(t\ne1\). The second gives \(a_3=t\) since \(b_1>0\). Thus \(b_1=(1-t)(2t-1)\), and the third equation becomes \(2v(3t-2)=0\). Hence \(t=2/3\), and

\begin{equation}\label{eq:small-rigid-A}
 a_0=1,\qquad a_i=\frac23\ (i>0),\qquad
 b_i=\frac19\ (0\le i<g-1).
\end{equation}

\emph{Genus six.}
The $\tau_1$-path gives nonzero entries at
\((1,4),(4,5),(2,5),(0,2),(0,3)\).
The vanishing entries give
\begin{gather*}
 a_1=a_4=t,\quad a_2=a_5=u,\quad a_3=2-t-u,\quad
 b_0=b_3=v,\quad b_1=b_4=w,\\
 t^2-t+v+w=0,\qquad
 tu-t+u^2-2u+b_2+1=0.
\end{gather*}
The $\tau_2$-path gives vanishing entries at \((1,5),(1,4),(0,3)\), respectively,
\[
 0=u-t,\qquad 0=w-v,\qquad 0=1-t-2v-w.
\]
Thus \(u=t\), \(w=v=(1-t)/3>0\), we have \((t-1)(t-2/3)=0\).
The root \(t=1\) would give \(v=0\), a contradiction. Therefore \(t=2/3\),
and we again have \eqref{eq:small-rigid-A}.

\emph{Computation of \(D_i\).}
Retain the notations
\[
    s=\sqrt{\frac{3}{4t_0}},
    \qquad
    v_i=su_i,
    \qquad
    J=\frac32A-I.
\]
Then \(J\) has diagonal \((1/2,0,\ldots,0)\), all its
adjacent off-diagonal entries equal \(1/2\), and
\[
    J=\Phi^{-1}\diag(2v_i^2-1)\Phi.
\]
This is precisely the Jacobi matrix appearing in Case C of the proof of Theorem~\ref{thm:weights}.
The characteristic-polynomial and recurrence computations for that matrix apply in these genera as well. They give
\[
    \{v_i:1\le i\le g\}
    =
    \left\{
      \cos\frac{(2k-1)\pi}{4g+2}:1\le k\le g
    \right\}
\]
and
\[
    \phi_r
    =\frac{(T_{2r+1}(v_i))_i}{\sqrt{\mathcal N}},
    \qquad
    \mathcal N=\sum_iD_iv_i^2.
\]
Since \(\Phi^{\mathsf T}D\Phi=I\), the square matrix
\(D^{1/2}\Phi\) also has orthonormal rows. Hence
\[
    D_i\sum_{r=0}^{g-1}T_{2r+1}(v_i)^2=\mathcal N.
\]
Writing \(v_i=\cos\theta_i\), we obtain
\[
    \sum_{r=0}^{g-1}\cos^2((2r+1)\theta_i)
    =\frac g2+
      \frac{\sin(4g\theta_i)}{4\sin(2\theta_i)}
    =\frac{2g+1}{4}.
\]
Therefore \(D_i=4\mathcal N/(2g+1)\) for every \(i\).
The relation \(D_i=m_1/m_i\) and primitivity give
\(m=(1,\ldots,1)\).
\end{proof}

The same recurrence gives
\[
    F_r(z)=\frac{T_{2r+1}(sz)}{T_{2r+1}(s)},
    \qquad 0\le r\le g-1.
\]
The field and Galois-action computations for Case C in the proof of
Theorem~\ref{thm:weights} therefore apply with
\(p=2g+1\in\{11,13\}\). They give
\[
    K=\Q(\zeta_p)^+,
    \qquad
    \operatorname{Gal}(K/\Q)
    \cong(\Z/p\Z)^\times/\{\pm1\},
    \qquad
    \tau_r\longmapsto\pm(2r+1).
\]
This completes the proof of Theorem~\ref{thm:weights}
in genera five and six.

\end{document}